\documentclass{daj}

\usepackage{amsmath,amssymb,amsthm}
\usepackage{graphicx}
\usepackage{hyperref}
\usepackage[hyphenbreaks]{breakurl}
\usepackage{cite}

\def\cC{\mathcal{C}}

\def\cF{\mathcal{F}}

\def\cP{\mathcal{P}}

\def\eps{\varepsilon}

\def\Z{\mathbb{Z}}
\def\R{\mathbb{R}}
\def\C{\mathbb{C}}

\def\F{\mathbb{F}}

\def\E{\mathbb{E}}

\def\nil{\mathrm{nil}}
\def\unf{\mathrm{unf}}
\def\sml{\mathrm{sml}}
\DeclareMathOperator\spa{span}

\DeclareMathOperator\mo{mod}

\newcommand{\vect}[1]{\boldsymbol{#1}}

\newtheorem{firstthm}{Proposition}[section]
\newtheorem{thm}[firstthm]{Theorem}
\newtheorem{prop}[firstthm]{Proposition}

\newtheorem{lemma}[firstthm]{Lemma}

\theoremstyle{definition}

\makeatletter
\newcommand{\subalign}[1]{%
  \vcenter{%
    \Let@ \restore@math@cr \default@tag
    \baselineskip\fontdimen10 \scriptfont\tw@
    \advance\baselineskip\fontdimen12 \scriptfont\tw@
    \lineskip\thr@@\fontdimen8 \scriptfont\thr@@
    \lineskiplimit\lineskip
    \ialign{\hfil$\m@th\scriptstyle##$&$\m@th\scriptstyle{}##$\hfil\crcr
      #1\crcr
    }%
  }%
}
\makeatother

\dajAUTHORdetails{%
  title = {On a conjecture of Gowers and Wolf}, 
  author = {Daniel Altman},
  plaintextauthor = {Daniel Altman},
}   

\dajEDITORdetails{%
   year={2022},
   number={10},
   received={30 June 2021},   
   published={12 September 2022},  
   doi={10.19086/da.38091},       
}   

\begin{document}

\begin{frontmatter}[classification=text]

\title{On a conjecture of Gowers and Wolf} 

\author[altman]{Daniel Altman}

\begin{abstract}
Gowers and Wolf have conjectured that, given a set of linear forms $\{\psi_i\}_{i=1}^t$ each mapping $\Z^D$ to $\Z$, if $s$ is an integer such that the functions $\psi_i^{s+1},\ldots, \psi_t^{s+1}$ are linearly independent, then averages of the form $\E_{\vect x} \prod_{i=1}^t f(\psi_i(\vect x))$ may be controlled by the Gowers $U^{s+1}$-norm of $f$. We prove (a stronger version of) this conjecture. 
\end{abstract}
\end{frontmatter}

\section{Introduction}\label{s:intro}
In \cite[Conjecture 2.5]{GW10}, Gowers and Wolf made a conjecture on the minimal Gowers norm of $f$ which is able to control averages of $f$ on a set of linear forms: $\E_{\vect x} \prod_{i=1}^t f(\psi_i(\vect x))$. In $\F_p^n$ ($p$ fixed, $n$ large), the conjecture was resolved by Gowers and Wolf for large enough $p$ \cite{GW10, GW11quad, GW11high} , and ultimately for all $p$ by Hatami--Hatami--Lovett \cite{HHL16}. Furthermore, Gowers and Wolf resolved the first nontrivial case of the conjecture in $\Z/N\Z$ in \cite{GW11cyc}. In the integers, the conjecture was considered to be resolved by Green and Tao until recently when it came to light (see \cite{GT20}, \cite{T20}) that the proof of \cite[Theorem 1.13]{GT10} requires an assumption on the system of linear forms in question.

Let $\Psi:=(\psi_i)_{i=1}^t$ be a family of linear forms, each mapping $\Z^D$ to $\Z$. For positive integers $k$, let $\Psi^{[k]}$ be the real vector space $\spa_{\R}\{(\psi_1(\vect x)^k, \ldots, \psi_t(\vect x)^k):\vect x \in \Z^D\} \leq \R^t$. The system of linear forms $\Psi$ satisfies \textit{the flag condition} if the containment of vector spaces $\Psi^{[k]} \leq \Psi^{[l]}$ holds whenever $k < l$.\footnote{We may also, at times, use \textit{flag} as an adjective.} 

Green and Tao's result \cite[Theorem 1.13]{GT10} holds for systems of linear forms that satisfy the flag condition. It is perhaps not \textit{a priori} clear to what extent this is a restriction on the space of linear forms and thus to what extent the Gowers-Wolf conjecture ought to be considered open. We begin with the modest remark that indeed there do exist systems of linear forms which do not satisfy the flag condition. Further, there are values of the parameters $(t,D)$ for which a system of linear forms on these parameters is generically not flag. On the other hand, many important examples of systems of linear forms do satisfy the flag condition. For example, systems corresponding to arithmetic progressions satisfy the flag condition. More generally, \textit{translation invariant} systems -- those for which the vector $(1,1,\ldots, 1)$ lies in  $\Psi^{[1]}$ -- satisfy the flag condition. For a more careful discussion of the significance of the flag condition, and counterexamples that lead to the restriction of \cite[Theorem 1.13]{GT10}, we direct the reader to \cite{GT20} and \cite{T20}.

The main result of this paper is Theorem \ref{t:main} below, a full resolution of the Gowers-Wolf conjecture; that is, a generalisation \cite[Theorem 1.13]{GT10} to systems of linear forms which do not necessarily satisfy the flag condition. Further to this, Theorem \ref{t:main} is stronger than an affirmative answer to the Gowers-Wolf conjecture in that it achieves control over averages involving $t$ distinct functions $\E_{\vect x \in [-N,N]^D} \prod_{i=1}^t f_i(\psi_i(\vect x))$  given the relatively weak information that only one of the functions $f_i$ has small $U^{s+1}$-norm.  A strengthening of the Gowers-Wolf conjecture in this direction has already been achieved in the finite field setting; see \cite{HL11}, \cite{HHL16}. Here and in what follows, $[N]:= \{1,\ldots,N\}$ and $[-N,N]:= \{ -N, \ldots, N\}$.\footnote{We work in $[-N,N]$ rather than $[N]$ because the set $K:=[-N,N]^D \cap \Psi^{-1}([-N,N]^t)$ satisfies $|K| \gg N^D$. This is not necessarily true of the set $[N]^D \cap \Psi^{-1}([N]^t)$, whereupon the upcoming theorem would be trivially true for those systems of linear forms with $|[N]^D \cap \Psi^{-1}([N]^t)| = o(N^D)$. Take, for example, the system $\Psi(x,y) = (-x,-x-y,-x-2y)$, or $\Psi'(x,y,\ldots) = (x-2y,y-2x,\ldots)$.}  

\begin{thm}\label{t:main}
Let $\Psi = (\psi_1, \ldots, \psi_t)$ be a collection of linear forms each mapping $\Z^D$ to $\Z$, and let $s \geq 1$ be an integer such that the polynomials $\psi_1^{s+1},\ldots,\psi_t^{s+1}$ are linearly independent. For $i=1, \ldots t$,  let $f_i:[-N,N] \to \C$ be functions bounded in magnitude by 1 (and defined to be zero outside of $[-N,N]$). For all $\eps >0$ there exists $\delta>0$ such that if $\min_{i} ||f_i||_{U^{s+1}[-N,N]} \leq \delta$, then 
\[\left|\E_{\vect x \in [-N,N]^D} \prod_{i=1}^t f_i(\psi_i(\vect x))\right| \leq \eps.\] 
\end{thm}

The strategy of proof employed is the same as the general `arithmetic regularity lemma and counting lemma' strategy put forward in \cite{GT10}. Unfortunately, the counting lemma \cite[Theorem 1.11]{GT10} applies only to systems of linear forms which satisfy the flag condition (indeed this is where \cite[Theorem 1.13]{GT20} fails to generalise). We need, in the first instance, a somewhat generalised and strengthened version of \cite[Theorem 1.13]{GT20} for flag systems of linear forms; this is Theorem \ref{t:flag} which is proven in Section 2. We defer to the beginning of Section 2 for a brief discussion on the ways in which Theorem \ref{t:flag} extends what is done in \cite[Theorem 1.13]{GT20}.  Then to recover Theorem \ref{t:main} above, we use Theorem \ref{t:flag} together with the observation that if $(\psi_i)_{i=1}^t$ is not flag, there is a sequence of integers $(a_i)_{i=1}^t$ such that the system $(a_i\psi_i)_{i=1}^t$ is flag; this is implemented in Section 3.

\subsection{Notation, conventions}\label{ss:not-conv}
The following notation and conventions are taken from related papers in the literature. 

As in \cite[Appendix B]{GT10L}, define Gowers norms on subsets of additive groups as follows. Let $Z$ be an abelian group and let $A\subset Z$. For any function $f:A\to \C$ and for all $s \geq 1$ define the Gowers uniformity norm $||f||_{U^{s+1}(A)}$ by 
\[||f||_{U^{s+1}(A)}^{2^{s+1}} = \E_{\subalign{&x, h: x+\omega \cdot h \in A\\ &\text{for all } \omega \in \{0,1\}^{s+1}}} \prod_{\omega \in \{0,1\}^{s+1}} \mathcal{C}^{|\omega|} f(x + \omega \cdot h)  ,\]
where the $x$ are elements of $Z$, the $h$ are vectors in $Z^{s+1}$, $\mathcal{C}$ is the complex conjugation operator and $|\omega| = \sum_{i=1}^{s+1} \omega_i$. As is noted in \cite[Appendix B]{GT10L}, we have the relation 
\[ ||f||_{U^{s+1}{(A)}} =  ||f1_A||_{U^{s+1}{(Z)}}/||1_A||_{U^{s+1}{(Z)}}.\]

As in \cite{GT10}, we will use $o_{A\to \infty; M}(X)$ as shorthand for a quantity bounded in magnitude  by  $c_M(A)X$ where $c_M(A)\to 0$ as $A \to \infty$ for $M$ fixed. We will also use slight variations on this notation which may be translated analogously.  

\section{A strengthened result in the flag case}
The goal of this section is to establish the following theorem. 

\begin{thm}\label{t:flag}
Let $\Psi = (\psi_1, \ldots, \psi_t)$ be a collection of linear forms each mapping $\Z^D$ to $\Z$ such that $\Psi$ satisfies the flag condition. Let $s \geq 1$ be an integer such that the polynomials $\psi_1^{s+1},\ldots,\psi_t^{s+1}$ are linearly independent. For $i=1, \ldots t$,  let $f_i:[N] \to \C$ be functions bounded in magnitude by 1 (and defined to be zero outside of $[N]$). Let $K$ be a subset of $[-N,N]^D$ whose convex hull in $\R^D$ is a polytope with $O_{\Psi,s,D,t}(1)$ faces, and for which $|K| =  \Theta_{\Psi,s,D,t}(N^D)$. Let $c \in \Z$ be such that $\Psi(K) + (c,c,\ldots, c) \subset [N]^t$. For all $\eps >0$ there exists $\delta>0$ (which depends on $D,t,s,\Psi$ and $\eps$, but is uniform in $N,K,f_i$ and $c$) such that if $\min_{i} ||f_i||_{U^{s+1}[N]} \leq \delta$, then 
\[\left|\E_{\vect x \in K} \prod_{i=1}^t f_i(\psi_i(\vect x) + c)\right| \leq \eps.\] 
\end{thm}

Theorem \ref{t:flag} is \cite[Theorem 1.13]{GT10} with some additional features and amendments which we will describe shortly. In particular our proof strategy is taken from \cite[Theorem 1.13]{GT10}. However, the modifications are essentially ubiquitous and so writing out the full proof appears to us to be necessary. 

The additional features in Theorem \ref{t:flag} which will be necessary to prove Theorem \ref{t:main} in the next section are as follows: 
\begin{enumerate}
\item we accommodate averages over large convex polytopes in $[-N,N]^D$,
\item we accommodate shifts of systems of linear forms, 
\item we obtain a result for averages comprising $t$ distinct functions $f_i$ given that only one of these has $U^{s+1}[N]$-norm less than $\delta$.
\end{enumerate}

In addition, we amend two parts of the argument in \cite[Theorem 1.13]{GT10}:
\begin{enumerate}
\item We give an explicit and quantitative argument for the case in which $N$ is bounded in terms of $\delta$; see Lemma \ref{l:smallN} below.
\item In the proof of \cite[Proposition 7.2]{GT10}, it is claimed that ``the contribution of $f_\sml(n)$ to (7.6) is $O(\eps)$ by the Cauchy-Schwarz inequality'', where (7.6) is an average over some  $\eps' N$-length subprogressions of $[N]$. It appears to us that one cannot reach this conclusion immediately and that indeed one needs to utilise outer averages over the set of $\eps'N$-length progressions.  The argument to bound the analogous term in our proof  is somewhat more intricate than an application of the Cauchy-Schwarz inequality and occurs in and around Lemma \ref{l:gtfix}.
\end{enumerate} 

\begin{proof}[Proof of Theorem \ref{t:flag}]
Recall the notation from the statement of Theorem \ref{t:flag}. Let $\eps >0$ and let $\delta>0$ be a parameter which we will of course optimise in terms of $\eps$ later. For $i=1, \ldots, t$ let  $\delta_i :=||f_i||_{U^{s+1}[N]}$  so that $\min_i\delta_i\leq  \delta$.  In this section  $\{t,D,s,\Psi\}$ is fixed, and  we will let all constants depend on this set without indicating this in our notation. Furthermore, we will abuse notation and let $C$ be a constant which may change line to line.

We may assume that $\eps$ is small because if the condition from Theorem \ref{t:flag} is true for $\eps$ then it is obviously true for all $\eps'>\eps$. To prove Theorem \ref{t:flag} it is then, of course, sufficient to show that 
\[\E_{\vect x \in K} \prod_{i=1}^t f_i(\psi_i(\vect x)+c) = O(\eps^{1/2}) ,\]
whenever $\delta$ is sufficiently small depending on $\eps$. 

Next, we will claim that we may assume that $N$ is large depending on $\eps$. 
\begin{lemma}\label{l:smallN}
The following inequality holds:
 \[ \left|\E_{\vect x \in K} \prod_{i=1}^t f_i(\psi_i(\vect x)+c)\right| \leq CN^{1/4}\delta.\]
\end{lemma}
\begin{proof}
Let $j$ be chosen so that $||f_j||_{U^{s+1}[N]} = \min_i ||f_i||_{U^{s+1}[N]} \leq \delta$. By Cauchy-Schwarz we may upper bound $|\E_{\vect x \in K} \prod_{i=1}^t f_i(\psi_i(\vect x)+c)|$ by $(\E_{\vect x \in K}|f_j(\psi_j(\vect x)+c)|^2)^{1/2}$. Next, $\#\{\vect x \in K : \psi_j(\vect x) + c = n\} = O(N^{D-1}) = O(|K|/N)$ and so $\E_{\vect x \in K}|f_j(\psi_j(\vect x)+c)|^2 \ll \E_{n \in [N]}|f_j(n)|^2$. Furthermore, observing that $|f_j(n)|^2 = f_j(n)\overline{f_j(n+0)}$ and then adding the (positive) contribution from all nonzero shifts $h$, we have
\begin{align*}
|\E_{n \in [N]}|f_j(n)|^2|^2 &\leq  \sum_{h\in \Z} |\E_{n \in [N]}f_j(n)\overline{f_j(n+h)}|^2\\
&= \sum_h \E_{n,m \in [N]}f_j(n)\overline{f_j(m)}\overline{f_j(n+h)}f_j(m+h)\\
&= \frac{1}{N^2} \sum_{n,h,h' \in \Z} f_j(n)\overline{f_j(n+h')}\overline{f_j(n+h)}f_j(n+h'+h)\\
&= \frac{\left|\{(n,h,h')\in \Z: n,n+h',n+h,n+h'+h \in [N]\}\right|}{N^2}||f_j||_{U^2[N]}^4.
\end{align*}
But $\#\{n,h,h'\in \Z: n,n+h',n+h,n+h'+h \in [N]\} = \Theta(N^3)$ and so we have that 
\[\left|\E_{n \in [N]}|f_j(n)|^2\right|^2 \leq CN||f_j||_{U^2[N]}^4. \]
Finally, it follows from the monotonicity of Gowers norms in additive groups and the formula $ ||f||_{U^{s+1}{(A)}} =  ||f1_A||_{U^{s+1}{(Z)}}/||1_A||_{U^{s+1}{(Z)}}$ which we recalled in Subsection \ref{ss:not-conv} above that $||f_j||_{U^2[N]} \ll ||f_j||_{U^{s+1}[N]}$ for $s \geq 1$. One concludes by combining the above inequalities. 
\end{proof}
It follows from the previous lemma that if $N = O_\eps(1)$ then we may simply choose $\delta$ to be sufficiently small depending on $N$ (depending on $\eps$) to prove Theorem \ref{t:flag}; the claim follows.

Let $s'$ be the Cauchy-Schwarz complexity of the system $\Psi$ (cf. \cite[Definition 1.3.2]{T12}). Apply Theorem \ref{t:arl} at step $s'$ with parameters $\eps$ as above and $\cF$ to be determined later to the functions  $\{  f_i\}_{i=1}^t$. We import notation from Appendix \ref{t:arl} and so in particular we have functions $ f_{i,\sml}$ with $|| f_{i,\sml}||_2 \leq \eps$, $ f_{i,\unf}$ with $|| f_{i,\unf}||_{U^{s'+1}[N]} \leq 1/\cF(M)$ and $ f_{i,\nil}$ with $ f_{i,\nil}(n) = F_i(g(n)\Gamma, n (\mo q), n/N)$ where $1 \leq q \leq M = O_{\eps, \cF}(1)$ and where $g(n)$ is an $(\cF (M),N)$-irrational polynomial sequence with respect to a filtered nilmanifold $(G/\Gamma, G_\bullet)$.

Then we may write \[ \E_{\vect x \in K}\prod_{i=1}^t  f_i(\psi_i(\vect x) + c) = \E_{\vect x \in K}\prod_{i=1}^t ( f_{i,\nil} +  f_{i,\sml} +  f_{i,\unf})(\psi_i(\vect x) + c)\] and expand this as the sum of $3^t$ terms each of the form $\E_{\vect x \in K}\prod_{i=1}^t  f_{i,\text{label}(i)}(\psi_i(\vect x) + c)$, where $\text{label}(i) \in \{\nil,\unf,\sml\}$ for each $i$. Recall that we inherit a quantity $M=O_{\eps,\cF}(1)$ from Theorem \ref{t:arl} which is an upper bound for the complexity of the nilsequences $f_{i,\nil}$. 

Of these $3^t$ terms, we claim that any term with some $\text{label}(i) = \sml$ is of size $O(\eps)$. Indeed by Cauchy-Schwarz we obtain the upper bound $(\E_{\vect x \in K} | f_{i,\sml}(\psi_i(\vect x) + c)|^2)^{1/2}$ for such a term. We conclude by recalling that $|K| = \Theta(N^D)$, noting that as $\vect x$ ranges over $K$, the quantity $\psi_i(\vect x) + c$ takes on any particular value in $[N]$ at most $N^{D-1}$ times, and using the $L^2$ bound inherited from Theorem \ref{t:arl}.

Next, any of these $3^t$ terms with some $\text{label}(i) = \unf$ is of size $o_{\cF(M) \to \infty}(1) + o_{N\to \infty; \cF(M)}(1)$ by Theorem~\ref{t:gvN}. Thus we have 
\begin{equation}\label{e:nilToMain}
\left|\E_{\vect x \in K}\prod_{i=1}^t  f_i(\psi_i(\vect x) + c)\right| \leq \left|\E_{\vect x \in K} \prod_{i=1}^t f_{i,\nil}( \psi_i(\vect x) + c)\right|+ O(\eps) + o_{\cF(M) \to \infty}(1) + o_{N\to \infty; \cF(M)}(1).
\end{equation}

Recall that we may write $ f_{i,\nil}(n) = F_i(g(n)\Gamma, n \ (\mo q), n/N)$. To deal with the mod $q$ and Archimedean behaviour, we will need to do some volume-packing. Let $\eps'$ be a small quantity depending on $\eps, M$; we will decide how small later. In what follows, a `cube' is a cartesian product of equal length intervals, and a `dilated cube' is a cartesian product of equal length equal step arithmetic progressions.

\begin{lemma}\label{l:packing}
There exists a `boundary' subset $S$ of $K$ which contains $O(q\eps'N^D)$ elements of $[N]^D$ such that the set $K\backslash S$ can be partitioned into cubes of side length $q\eps'N$, where each cube is itself a disjoint union of $q^D$ dilated cubes of the form $P_1 \times \cdots \times P_D$ where each $P_i$ is an arithmetic progression in $[N]$ of spacing $q$ and length $\eps'N$.
\end{lemma} 
\begin{proof}
The decomposition of cubes with $q$-divisible side lengths into $q^D$-many $q$-dilated cubes is obvious. From here the strategy is essentially to draw $K$ on $D$-dimensional grid paper with cubes of side length $q \eps'N$ and exclude from $K$ any cubes which are not strictly contained in $K$. 
Let $Q$ be  a cube such that $Q \cap K \ne \emptyset$ and $Q \cap K^c \ne \emptyset$. 
The maximum distance between any two points in $Q$ is $\sqrt{D}q\eps'N$ and so all points in $Q$ lie within the $\sqrt{D}q\eps'N$-neighbourhood of the boundary of (the convex hull in $\R^D$ of) $K$. Now recall that this boundary comprises at most $O(1)$ codimension 1 faces. The $\sqrt{D}q\eps'N$-neighbourhood of any one of these faces may contain at most $O(q\eps'N^D)$ points in $[-N,N]^D$. It follows that there are at most $O(q\eps'N^D)$ points in $K$ which lie in the $\sqrt{D}q\eps'N$-neighbourhood of the boundary of $K$.
\end{proof}

Remove $S$ from $K$ and partition $K\backslash S$ as per the above lemma. Let $\cP$ be the collection of all $P:= P_1 \times \cdots \times P_D$. Then we have 
\begin{equation}
\left|\E_{\vect x \in K} \prod_{i=1}^t  f_{i,\nil}(\psi_i(\vect x) + c)\right| \leq \left|\E_{P\in \cP}\E_{\vect x \in P} \prod_{i=1}^t f_{i,\nil}(\psi_i(\vect x) + c)\right| + O_M(\eps').
\end{equation}

Observe that for a fixed $P$, all $\vect x \in P$ yield the same value $b_{P,i}:= \psi_i(\vect x) + c \ (\mo q)$. Furthermore, there is a number $c_{P,i}$ such that $|(\psi_i(\vect x) + c)/N-c_{P,i}| = O(\eps')$ for all $\vect x \in P$. Thus from the fact that $F_i$ is $M$-Lipschitz we have   
\begin{equation}
\E_{P\in \cP}\E_{\vect x \in P} \prod_{i=1}^t f_{i,\nil}( \psi_i(\vect x) + c) = \E_{P\in \cP}\E_{\vect x \in P} \prod_{i=1}^t F_i(g( \psi_i(\vect x) + c)\Gamma,b_{P,i},c_{P,i}) + O_M(\eps').
\end{equation}

Now define $\tilde g(n) = g(n + c)$; \cite[Lemma A.8]{GT10} gives that $\tilde g$ is also $(\cF(M), N)$-irrational. Also recall the definition of the Leibman group $G^\Psi:= \langle g_i^{v_i}: g_i \in G_i, v_i \in \Psi^{[i]} \rangle$ (we direct the reader to \cite[Chapter 3]{GT10} for some basic facts pertaining to $G^\Psi$). Finally, recall that the linear forms $\Psi$ satisfy the flag condition. Thus we may apply Theorem \ref{t:cl} to obtain 
\begin{align}\label{e:cl}
\E_{P\in \cP}\E_{\vect x \in P} \prod_{i=1}^t  F_i(\tilde g(\psi_i(\vect x))\Gamma,b_{P,i},c_{P,i}) &= \E_{P\in \cP}\int_{(g_1,\ldots, g_t)\Gamma^\Psi \in G^{\Psi}/\Gamma^{\Psi}} \prod_{i=1}^t F_i(g_i\Gamma, b_{P,i}, c_{P,i})\\
&\quad + o_{\cF(M)\to \infty;M}(1) + o_{\eps'N \to \infty; M}(1).\nonumber
\end{align}

Since the forms $\psi_1^{s+1}, \ldots, \psi_t^{s+1}$ are linearly independent, we have $ \Psi^{[s+1]} = \R^t$ and so $G_{s+1}^t \leq G^{ \Psi}$. Now, recall from \cite[Chapter 7]{GT10} the notation \[F_{i,\leq s}(g\Gamma,b,c) := \int_{G_{s+1}/\Gamma_{s+1}} F_i(gg_{s+1}\Gamma,b,c)dg_{s+1}.\] It follows that 
\[\int_{G^{ \Psi}/\Gamma^{ \Psi}} \prod_{i=1}^t F_i(g_i\Gamma, b_{P,i}, c_{P,i}) = \int_{G^{ \Psi}/\Gamma^{ \Psi}} \prod_{i=1}^t F_{i,\leq s}(g_i\Gamma, b_{P,i}, c_{P,i}).\]
Furthermore, 
\begin{align*}
\left|\int_{G^{ \Psi}/\Gamma^{ \Psi}} \prod_{i=1}^t F_{i,\leq s}(g_i\Gamma, b_{P,i}, c_{P,i})\right| &\leq \int_{G^{ \Psi}/\Gamma^{ \Psi}} \prod_{i=1}^t \left|F_{i,\leq s}(g_i\Gamma, b_{P,i}, c_{P,i})\right| \\
&\leq \int_{G^{ \Psi}/\Gamma^{ \Psi}} \min_i|F_{i,\leq s}(g_i\Gamma, b_{P,i}, c_{P,i})| \\ &\leq \min_i \int_{G^/\Gamma} |F_{i,\leq s}(g\Gamma, b_{P,i}, c_{P,i})| \\
&\leq \min_i \left(\int_{G^/\Gamma} |F_{i,\leq s}(g\Gamma, b_{P,i}, c_{P,i})|^2\right)^{1/2}
\end{align*}
where we have used in the penultimate line that the projection of $G^{ \Psi}$ onto any of its coordinates is surjective onto $G$ (since each $\psi_i$ must be nonzero) and that this projection maps the Haar measure on $G^\Psi$ to the Haar measure on $G$. 

All in all we have (continuing on from (\ref{e:cl})):
\begin{equation}\label{e:projbd}
\left|\E_{P\in \cP}\int_{G^{ \Psi}/\Gamma^{ \Psi}} \prod_{i=1}^t F_i(g_i\Gamma, b_{P,i}, c_{P,i})\right| \leq \min_i \E_{P\in \cP} \left(\int_{G^/\Gamma} |F_{i,\leq s}(g\Gamma, b_{P,i}, c_{P,i})|^2\right)^{1/2}.
\end{equation}

For brevity we temporarily abuse notation and use $F_{i}$ to denote the function $F_i(\cdot, b_{P,i},c_{P,i})$; similarly for  $F_{i,\leq s}$. We have 
\[\int_{G/\Gamma} |F_{i,\leq s}|^2 = \int_{G/\Gamma}(F_{i,\leq s} - F_i)\overline{F_{i,\leq s}} + \int_{G/\Gamma}F_i\overline{F_{i,\leq s}} = \int_{G/\Gamma}  F_i\overline{F_{i,\leq s}}, \] 
since $F_{i,\leq s}$ is invariant on $G_{s+1}$ cosets and $F_i-F_{i,\leq s}$ integrates to zero on any such coset. 

\begin{lemma}\label{l:tildeP}
For each $P \in \cP, i=1,\ldots, t$ there is a subprogression $\tilde P_i \subset \psi_i(P)+c$ of size  $\Theta(\eps' N)$ such that $n  = b_{P,i} \ (\mo q)$ and $|n/N - c_{P,i}|=O(\eps')$ for all $n \in \tilde P_i$.
\end{lemma}
\begin{proof}
Recall that $b_{P,i}$ was defined to be the (unique) residue class of $\psi_i(P) + c \ (\mo q)$ and that $c_{P,i}$ was defined to be a number such that $|(\psi_i(\vect x) + c)/N-c_{P,i}| = O(\eps')$ for all $\vect x \in P$.  Thus we may let $\tilde P_i$ be \textit{any} subprogression of length $\Theta(\eps' N)$ contained in the set $\psi_i(P)+c$. Note that $|\psi_i(P) + c| = O(\eps'N)$ and so we need only concern ourselves with a lower bound on the subset we seek. To see that $\psi_i(P)+c$ does contain a subprogression of length $\gg \eps' N$, let $j\in [D]$ be an index such that the $j$th coefficient of $\psi_i$ is nonzero. Recall that $P=P_1 \times \cdots \times P_D$ is a (dilated) cube in which each side contains $\eps' N$ points, and so we may let $\tilde P_i$ be the image under $\psi_i(\cdot) + c$ of the set $(p_0,p_1,\ldots, P_j, \ldots, p_D)$ where each of the $p_k$ are fixed points in $P_k$ for $k \ne j$.
\end{proof}

Let $\tilde P_i$ be as in the previous lemma. Then we may apply Theorem \ref{t:cl} to obtain 
\[ \int_{G/\Gamma}  F_i\overline{F_{i,\leq s}} = \E_{n\in \tilde P_i} F_i\overline{F_{i,\leq s}}( g(n) \Gamma, b_{P,i},c_{P,i}) + o_{\cF(M) \to \infty ; M}(1) + o_{\eps'N\to \infty;M}(1). \]
Recalling the properties of $\tilde P_i$ from Lemma \ref{l:tildeP} and that $F_i$ (viewed as a function on $G/\Gamma \times  \Z/q\Z \times \R \to \C$, i.e. ceasing the abuse of notation) has Lipschitz constant $\leq M$, we may write the right hand side of the above as 
\[ \E_{n\in \tilde P_i} F_i( g(n) \Gamma, n \ (\mo q), n/N)\overline{F_{i,\leq s}}( g(n) \Gamma, b_{P,i},c_{P,i}) + o_{\cF(M) \to \infty ; M}(1) + o_{\eps'N\to \infty;M}(1) + O_M(\eps'). \]
Thus since (\ref{e:projbd}) we have shown: 
\begin{align}\label{e:progAgain}
 & \min_i \E_{P\in \cP} \left(\int_{G^/\Gamma} |F_{i,\leq s}(g\Gamma, b_{P,i}, c_{P,i})|^2\right)^{1/2} \nonumber \\ &\qquad \leq \min_i \E_{P \in \cP}\left|\E_{n\in \tilde P_i} F_i( g(n) \Gamma, n \ (\mo q), n/N)\overline{F_{i,\leq s}}( g(n) \Gamma, b_{P,i},c_{P,i})\right|^{1/2}\\ \nonumber
 & \qquad \qquad + o_{\cF(M) \to \infty ; M}(1) + o_{\eps'N\to \infty;M}(1) + O_M(\eps').
\end{align}

Now recall that $F_i( g(n) \Gamma, n\ (\mo q), n/N) =  f_i(n) -  f_{i,\sml}(n) -  f_{i,\unf}(n)$. Then substituting $F_i( g(n) \Gamma, n\ (\mo q), n/N) =  f_i(n) -  f_{i,\sml}(n) -  f_{i,\unf}(n)$ into (\ref{e:progAgain}), we claim that the term with $f_i$ is of size $o_{\delta_i\to 0;M,\eps'}(1)+ o_{N\to \infty;M,\eps'}(1)$. Indeed, by Proposition \ref{p:gowersUniformSubprog} we have that $||f_i||_{U^{s+1}(\tilde P_i)} = O_{N\to \infty;\eps'}(\delta_i)+ o_{N\to \infty;\eps'}(1)$. Also, observe that $\overline{F_{i,\leq s}}( g(n) \Gamma, b_{P,i},c_{P,i})$ is an $s$-step nilsequence of complexity $O_M(1)$.  Thus our claim follows by invoking the converse to the inverse theorem for Gowers norms on $\tilde P_i$. Similarly, we may conclude that the term with $ f_{i,\unf}$ is of size $o_{\cF(M) \to \infty; M, \eps'}(1) + o_{N\to \infty;M,\eps'}(1)$ by the same argument and the additional ingredients that $s' \geq s$ and that Gowers norms are monotonic (up to a constant factor) in $s$.

For the term with $f_{i,\sml}$ we need to utilise the average over $P \in \cP$. First we compute  the following using the Cauchy-Schwarz inequality and convexity: 
\begin{align*}
\min_i \E_{P \in \cP}\left|\E_{n\in \tilde P_i} f_{i,\sml}(n)\overline{F_{i,\leq s}}( g(n) \Gamma, b_{P,i},c_{P,i})\right|^{1/2} &\leq \min_i \E_{P \in \cP}\left(\E_{n\in \tilde  P_i} | f_{i,\sml}(n)|^2\right)^{1/4} \\
&\leq \min_i \left(\E_{P \in \cP}\E_{n\in \tilde P_i} | f_{i,\sml}(n)|^2\right)^{1/4}.
\end{align*}
Now we need the following lemma which implies that as $P$ varies in $\cP$, the progressions $\tilde P_i$ are sufficiently well distributed in $[N]$. 

\begin{lemma}\label{l:gtfix}
For each $n \in [N]$, there are at most $O(\eps'^{-(D-1)})$ elements $P \in \cP$ such that $n \in \psi_i(P) + c$.
\end{lemma}
\begin{proof}
One may assume that $n \in \psi_i(K \backslash S) + c$ because otherwise the statement is trivial. Let $T:=\psi_i^{-1}(n-c)$. 
 Any cube (as in the statement of Lemma \ref{l:packing}) which contains a point in $T$ lies entirely in the $\sqrt{D}q \eps'N$-neighbourhood of $T$; denote this neighbourhood by $U$. Note that $U$ has volume $O(q\eps' N^D)$. On the other hand, the volume of one of these cubes is $(q\eps'N)^D$, and so $U$ can contain at most $O(\eps'^{-(D-1)}/q^{D-1})$ distinct cubes. Fix such a cube $Q$ and recall that $Q$ comprises $q^{D}$ different $P \in \cP$. Let $\cP_Q$ be those $P\in \cP$ which are contained in $Q$; it remains to argue that at most $O(q^{D-1})$ elements of $\cP_Q$ can contain a point in $T$. 

Note that $P\cap T \ne \emptyset$ implies that $\psi_i(P)+c = n \ (\mo q)$, which is well-defined by the $q$-periodicity of $P$ in all coordinate directions. Thus it suffices to bound the number of $P\in \cP_Q$ with $\psi_i(P)= n-c \ (\mo q)$. Every $P\in \cP_Q$ has exactly one representative in a fundamental domain $F$ for $\Z^D/(q\Z)^D$ which is chosen to be contained in $Q$. Furthermore the map $\psi_i : \Z^D \to \Z/q\Z$ is a group homomorphism which factors through the quotient $\Z^D/(q\Z)^D$; denote the corresponding map from $\Z^D/(q\Z)^D$ to $\Z/q\Z$ by $\overline{\psi_i}$. Thus we may view $\overline{\psi_i}$ as a group homomorphism on $F$ and so the number of $P\in \cP_Q$ with $\psi_i(P) = n-c \ (\mo q)$ is either 0 or $|\ker \overline{\psi_i}|$. But $|\ker \overline{\psi_i}| $ is easily seen to be of size $O(q^{D-1})$ by invoking the fact that the size of the coefficients of $\psi_i$ are $O(1)$. 
\end{proof}

Since $|K|\gg N^D$ and recalling Lemma \ref{l:packing}, we have $|\cP| \geq C\eps'^{-D}$. Also, $|\tilde P_i| = \Theta(\eps'N)$, and so using the previous lemma in the penultimate line, 
\begin{align*}
\E_{P \in \cP}\E_{n\in \tilde  P_i} | f_{i,\sml}(n)|^2 &\leq \frac{\eps'^{(D-1)}}{CN}\sum_{P\in \cP,n\in \tilde P_i}| f_{i,\sml}(n)|^2 \\
&= \frac{\eps'^{(D-1)}}{CN}\sum_{n \in [N]}| f_{i,\sml}(n)|^2 \#\{P \in \cP: n \in \tilde P_i\} \\
&\leq  \frac{1}{C}|| f_{i,\sml}||_2^2\\
&= O(\eps^2).
\end{align*}

All in all, since (\ref{e:progAgain}) we have shown that  
\begin{align}\label{e:finalBds}
&\min_i \E_{P \in \cP}\left|\E_{n\in \tilde P_i} F_i( g(n) \Gamma, n \ (\mo q), n/N)\overline{F_{i,\leq s}}( g(n) \Gamma, b_{P,i},c_{P,i})\right|^{1/2} \nonumber \\& \qquad \qquad \leq \min_i o_{\delta_i\to 0;M,\eps'}(1) + o_{\cF(M) \to \infty; M,\eps'}(1) + o_{N\to \infty;M,\eps'}(1)  + O(\eps^{1/2}).
\end{align}

Recall that $\min_i \delta_i \leq \delta$  so that as $\delta \to 0$ we have $\min_i \delta_i \to 0$ and so we may write  $\min_i o_{\delta_i\to 0;M,\eps'}(1) = o_{\delta \to 0;M,\eps'}(1)$. Putting together the results from Equations (\ref{e:nilToMain})--(\ref{e:finalBds}), we have
\begin{align*}
\left|\E_{\vect x \in K}\prod_{i=1}^t  f_i(\psi_i(\vect x) + c)\right|&\leq  O(\eps)  + o_{N\to \infty; \cF(M)}(1) + O_{M}(\eps') + o_{\cF(M)\to \infty;M,\eps'}(1)\\ &\qquad   + o_{\eps'N\to \infty;M}(1) + o_{\delta \to 0;M,\eps'}(1) + o_{N\to \infty;M,\eps'}(1) + O(\eps^{1/2}).
\end{align*}
We complete the proof of Theorem \ref{t:flag} by choosing $\eps'$ to be sufficiently small depending on $M$ and $\eps$, choosing $\cF$ to be a sufficiently rapidly growing function depending on $\eps'$ and the function implicit in the $o_{\cF(M)\to \infty;M,\eps'}(1)$ notation,  choosing $\delta$ to be sufficiently small depending on $M$ and $\eps'$  and setting $N$ to be sufficiently large depending on $\eps'$, $M$, $\cF$ and $\eps$. 
\end{proof}

\section{Proof of Theorem \ref{t:main}}
In what follows we will use notation from the statement of Theorem \ref{t:main} without explicitly reintroducing it. In particular, in this section $\Psi= (\psi_1,\ldots, \psi_t)$ is a system of linear forms which is not necessarily flag.  As in the previous section, we will allow all constants to depend on $\Psi, D, s, t$ and let $C$ be a constant which may change line to line. Note also that we may take $N$ arbitrarily large in terms of $\eps$ because the case where $N$ is bounded in terms of $\eps$ may be handled with exactly the same argument as it was in the proof of Theorem \ref{t:flag}.

We claim that there is a sequence of scalars $(a_i)_{i=1}^t$ each in $\Z\backslash \{0\}$ and of size $O(1)$ such that the family $\tilde \Psi := (a_i  \psi_i)_{i=1}^t$ satisfies the flag property. Indeed, since all $\psi_i$ are nonzero, the image of $\Psi$ is a linear space which is not contained in any hyperplane $x_i=0$, and so it contains a vector which is nonzero in every component; denote a choice of such a vector by $(b_1, \ldots, b_t)$. Then setting $a_i = \prod_{j \ne i} b_i$ ensures that the image of $\tilde \Psi := ( a_i  \psi_i)_{i=1}^t$ contains a nonzero scalar multiple of $(1,1,\ldots, 1) \in \R^t$ so that $\tilde \Psi$ is translation invariant and hence satisfies the flag property. 

Let $(a_i)_{i=1}^t$ be integers as above and let $a = \max_i |a_i|$. Define $\tilde f_i$ on $a_i[-N,N]+aN$ by $\tilde  f_i(x) = f_i((x-aN)/a_i)$. Extend $\tilde f_i$ to a 1-bounded function on $[2aN]$ by letting $\tilde f_i(x) = 0$ for $x \in [2aN]\backslash (a_i[-N,N]+aN)$. 

\begin{lemma}\label{l:gowNorms}
The following inequality holds:
\[ ||\tilde f_i||_{U^{s+1}[2aN]} \leq || f_i||_{U^{s+1}[-N,N]} + o_{N\to \infty}(1).\]
\end{lemma}
\begin{proof}
We compute directly from the definition of Gowers norms (cf. Section \ref{ss:not-conv} above), recalling that $\tilde f_i(x) = 0$ for $x \in [2aN]\backslash (a_i[-N,N]+aN)$: 
\begin{align*}
||\tilde f_i||_{U^{s+1}[2aN]}^{2^{s+1}} &= \E_{\subalign{&x, h: x+\omega \cdot h \in [2aN]\\ &\text{for all } \omega \in \{0,1\}^{s+1}}} \prod_{\omega\in\{0,1\}^{s+1}} \cC^{|\omega|} \tilde f_i(x + \omega \cdot h) \\
&=  \frac{\left|\{(x,h):x + \omega \cdot h \in a_i[-N,N]+aN \text{ for all } \omega \in \{0,1\}^{s+1} \}\right|}{\left|\{(x,h):x + \omega \cdot h \in [2aN] \text{ for all } \omega \in \{0,1\}^{s+1}\}\right|} \times \\
& \qquad \E_{\subalign{&x, h: x+\omega \cdot h \in a_i[-N,N]+aN\\ &\text{for all } \omega \in \{0,1\}^{s+1}}}\prod_{\omega\in\{0,1\}^{s+1}} \cC^{|\omega|} \tilde f_i(x + \omega \cdot h) \\
& \leq  \E_{\subalign{&x, h: x+\omega \cdot h \in a_i[-N,N]+aN\\ &\text{for all } \omega \in \{0,1\}^{s+1}}}\prod_{\omega\in\{0,1\}^{s+1}} \cC^{|\omega|} \tilde f_i(x + \omega \cdot h)  + o_{N\to \infty}(1)\\
&= ||\tilde f_i||_{U^{s+1}(a_i[-N,N]+aN)}  + o_{N\to \infty}(1). 
\end{align*}
But one sees easily that $||\tilde  f_i||_{U^{s+1}(a_i[-N,N]+aN)} = || f_i||_{U^{s+1}[-N,N]}$ since $f_i = \tilde f_i \circ \phi$ where $\phi(x) = a_ix+aN$, which is a Freiman isomorphism on the respective sets.
\end{proof}

Let $K = [-N,N]^D \cap \Psi^{-1}([-N,N]^t)$. Then we have 
\begin{align}\label{e:main}
\left|\E_{\vect x \in [-N,N]^D} \prod_{i=1}^t f_i(\psi_i(\vect x))\right| &\leq  \left|\E_{\vect x \in K}\prod_{i=1}^t f_i(\psi_i(\vect x))\right| \nonumber  \\
&=  \left|\E_{\vect x \in K}\prod_{i=1}^t \tilde f_i(a_i\psi_i(\vect x) + aN)\right|.
\end{align}
By insisting that $\min_i||f_i||_{U^{s+1}[-N,N]}\leq \delta $ is sufficiently small and $N$ is sufficiently large, we may force $\min_i||\tilde f_i||_{U^{s+1}([2aN])}$ to be as small as we like by Lemma \ref{l:gowNorms}. Now, the system of linear forms $\tilde \Psi = (a_i\psi_i)_{i=1}^t$ is flag and furthermore $(a_1\psi_1)^{s+1}, \ldots, (a_t\psi_t)^{s+1}$ are linearly independent since $(\psi_1^{s+1}, \ldots, \psi_t^{s+1})$ are. Furthermore, one may check that the volume of $K$ is equal to $|K| + O(N^{D-1})$ (in fact, this statement is proven in \cite[Appendix A]{GT10L} for $K$ convex). Also, one may show easily that the volume of $K$ is of size  $\Theta(N^D)$. Thus we may invoke Theorem \ref{t:flag} (with functions on $[2aN]$ rather than $[N]$) to force $|\E_{\vect x \in K}\prod_{i=1}^t \tilde f_i(a_i\psi_i(\vect x) + aN)|$ to be as small as we like. In turn we may make $|\E_{\vect x \in [-N,N]^D} \prod_{i=1}^t f_i(\psi_i(\vect x))|$ as small as we like from Equation (\ref{e:main}). This completes the proof of Theorem \ref{t:main}.


\appendix 
\section*{Appendix}\section{Results from higher order Fourier analysis}
In this appendix we collect some results from higher order Fourier analysis which we need at various points in the paper. All are known to experts and/or are small perturbations of results that already appear in the literature.

\begin{thm}[Simultaneous arithmetic regularity lemma]\label{t:arl}
Let $f_1, \ldots, f_t: [N] \to [0,1]$ be functions, let $s \geq 1$ be an integer, let $\eps > 0$, and let $\cF: \R^+ \to \R^+$ be a growth function. Then there exists a quantity $M=O_{s,\eps,\cF,t}(1)$ and  decompositions 
\[f_i = f_{i,\nil} + f_{i,\sml} + f_{i,\unf} \]
of each $f_i$ into functions $f_{i, \nil}, f_{i, \sml}, f_{i,\unf}: [N] \to [-1,1]$ where 
\begin{enumerate}
\item($f_{i,\nil}$ structured) $f_{i,\nil}$ is a $(\cF(M), N)$-irrational virtual nilsequence of degree $\leq s$, complexity $\leq M$ and scale $N$, and we may write $f_{i,nil}(n) = F_i(g(n) \Gamma, n \ (\mo q), n/N)$, where $g$ is a polynomial sequence adapted to a filtered nilmanifold and $\Gamma$ is a cocompact lattice on that nilmanifold,
\item($f_{i,\sml}$ small) $f_{i,\sml}$ has an $L^2[N]$-norm of at most $\eps$,
\item($f_{i,\unf}$ very uniform) $f_{i,\unf}$ has a $U^{s+1}[N]$-norm of at most $1/\cF(M)$,
\item(Nonnegativity) $f_{i,\nil}$ and $f_{i,\nil} + f_{i,\sml}$ take values in $[0,1]$.
\end{enumerate}
In particular, the  functions $\{f_{i,\nil}\}_{i=1}^t$ differ only by the choice of Lipschitz function $\{F_i\}_{i=1}^t$. 
\end{thm}
\begin{proof}
This is \cite[Theorem 1.2]{GT10} together with the extra ingredient that the nilsequences $f_{i,\nil}$ differ only by the choice of Lipschitz function. The existence of  this result is well-known to experts. 

One derives the above statement from \cite{GT10} by invoking the non-irrational regularity lemma \cite[Proposition 2.7]{GT10} at each $i$, forming the product polynomial sequence $g(n) = (g_1(n), \ldots, g_t(n))$ on the product nilmanifold $\prod_{i=1}^t G_i/\prod_{i=1}^t \Gamma_i$ and then factorising for irrationality as in \cite[Chapter 2]{GT10}. We omit the details but remark that they are worked out in a similar setting in \cite[Appendix A]{K20}.
\end{proof}

We refer the reader to \cite{GT10} for further background on the theorem that follows. 

\begin{thm}[Flag counting lemma, {\cite[Theorem 1.11]{GT10}}]\label{t:cl}
Let $M, D, t, s$ be positive integers with $D,t,s \leq M$, let $(G/\Gamma, G_\bullet)$ be a degree $\leq s$ filtered nilmanifold of complexity $\leq M$,  let $g: \Z \to G$ be an $(A,N)$-irrational polynomial sequence adapted to $G_\bullet$, let $\Psi = (\psi_1,\ldots, \psi_t)$ be a system of linear forms  with coefficients of magnitude at most $M$ and which satisfies the flag condition, let $P$ be a convex subset of $[-N,N]^D$ with positive Lebesgue measure, let $\Lambda \leq \Z^D$ be a sublattice of index $[\Z^D:\Lambda]\leq M$ and let $\vect x_0 \in \Z^D$. Then for any sequence of $M$-Lipschitz functions $F_1,\ldots,F_t: G/\Gamma \to \C$, we have 
\[\E_{\vect x \in (\vect x_0 + \Lambda )\cap P} \prod_{i=1}^t F_i(g(\psi_i(\vect x))\Gamma) =  \int_{g(0)^\Delta G^\Psi/\Gamma^\Psi}\otimes_{i=1}^t F_i + o_{A\to \infty;M}(1) + o_{N \to \infty;M}(1),\]
where $g(0)^\Delta:= (g(0),\ldots, g(0)) \in G^t$.
\end{thm}

Recall the definition of the \textit{Cauchy-Schwarz complexity} of a system of linear forms (cf. \cite[Definition 1.3.2]{T12}). 

\begin{thm}[Generalised von Neumann inequality on convex sets]\label{t:gvN}
Let $t,D$ be positive integers and let $L>0$. Let $f_1, \ldots, f_t:[N] \to \C$ be 1-bounded. Let $\Psi = (\psi_1, \ldots, \psi_t)$ be a system of affine-linear forms with constant term of size at most $LN$ and linear coefficients of size at most $L$. Let $K \subset [N]^D$ be convex and such that $\Psi(K) \subset [N]^t$. Then letting $s$ be the Cauchy-Schwarz complexity of $\Psi$ and letting $\delta$ be such that $\min_{i=1}^t||f_i||_{U^{s+1}[N]} \leq \delta$ we have
\[ \E_{\vect x \in K} \prod_{i=1}^t f_i(\psi_i(\vect x)) = o_{\delta \to 0;L}(1) + o_{N \to \infty;\delta,L}(1).\] 
\end{thm}
\begin{proof} This is essentially \cite[Proposition 7.1]{GT10L}, except we are dealing with functions $f_i$ which are 1-bounded rather than bounded by a pseudorandom measure. The proof of  \cite[Proposition 7.1]{GT10L} which is given in  \cite[Appendix C]{GT10L} is easily adapted to recover the statement above.
\end{proof}

\begin{prop}[Gowers uniformity on subprogressions]\label{p:gowersUniformSubprog}
Let $f:[N]\to \C$ be 1-bounded, let $\eps >0$ and let $P$ be a subprogression of $[N]$ of size $\eps N$. Then \[||f||_{U^{s+1}(P)} = O_{N \to \infty;\eps}(||f||_{U^{s+1}[N]}) + o_{N \to \infty;\eps}(1).\]
\end{prop}
\begin{proof}
In this proof, all asymptotic notation is implicitly with respect to the limit $N \to \infty$. From Subsection \ref{ss:not-conv} we have $||f||_{U^{s+1}(P)} = \Theta_\eps(||1_Pf||_{U^{s+1}[N]})$. Let $I_\eps = [-\eps N/2,\eps N/2]$. Write $P=qI_\eps+c$. Let $\chi_0$ be the function $\Z \to \R$ which is $1$ on $I_\eps$, which is zero when $|x| > (1+1/\sqrt {N})\eps N/2$, and which linearly interpolates otherwise ($\chi_0$ is the Fourier transform of a de la Vall\'ee Poussin kernel). It is standard that $||\hat \chi_0||_1 = O_\eps(1)$ (one may see this by writing $\chi_0$ as the convolution of suitably normalised characteristic functions of intervals).   
Define $\chi$ by $\chi(x) = \chi_0((x-c)/q)$ for all $x = c\ (\mo q)$ and zero otherwise. Then we have $||\hat \chi||_1 = O_\eps(1)$ and furthermore that $1_P-\chi$ is supported on a set of size $o_\eps(N)$. The former property and Fourier inversion yield that 
\[ ||\chi f||_{U^{s+1}[N]} \leq ||\hat \chi||_1\sup_\theta||e(\theta \cdot)f||_{U^{s+1}[N]} = O_\eps(||f||_{U^{s+1}[N]}),\]
using that the $U^{s+1}[N]$-norm is linear phase invariant (see \cite[Exercise 1.3.21]{T12}). The latter property yields that $||(1_P-\chi)f||_{U^{s+1}[N]} = o_\eps(1)$. We are ready to conclude: 
\[ ||1_Pf||_{U^{s+1}[N]} \leq ||\chi f||_{U^{s+1}[N]} + ||(1_P-\chi)f||_{U^{s+1}[N]} = O_{\eps}(||f||_{U^{s+1}[N]}) + o_{\eps}(1).\]
\end{proof}

\section*{Acknowledgments} 
The author is grateful to both Ben Green and Freddie Manners for helpful conversations and for valuable feedback on earlier versions of this document.

\bibliographystyle{amsplain}


\begin{dajauthors}
\begin{authorinfo}[pgom]
  Daniel Altman\\
  University of Oxford\\
  Oxford, United Kingdom\\
  daniel.h.altman\imageat{}gmail\imagedot{}com \\
\end{authorinfo}
\end{dajauthors}

\end{document}